\documentclass{amsart}
\usepackage[pagewise]{lineno}
\usepackage{amsmath,amssymb}

\theoremstyle{remark}

\numberwithin{equation}{section}

\input{tcilatex}

\begin{document}
\title[On $SQT$-groups]{Finite groups in which strong $4$-quasinormality is
transitive relation. }
\author{Khaled A. Al-Sharo}
\address{Department of Mathematics, Al al-Bayt University, Mafraq-25113,
JORDAN. }
\email{sharo\_kh@yahoo.com}
\thanks{This paper is in final form and no version of it will be submitted
for publication elsewhere.}
\date{}
\dedicatory{}
\maketitle

\begin{center}
\today
\end{center}

\begin{quote}
{\small \textsc{Abstract.} }Two subgroups $H$ and $K$ are $4$-permutable in $%
G$ if $\left\langle H\cup K\right\rangle =HKHK$, and $H$ is strong $4$%
-quasinormal in $G$ if $H$ is $4$-permutable with every subgroups $K$ of $G$%
. A fnite group $G$ is called $Sq4T$-group if strong $4$-quasinormality is
transitive relation among the subgroup of $G$. In this paper we studyfnite $%
Sq4T$-groups. In particular, we prove that every $Sq4T$-group is solvable.
Some facts related to $4$-permutability and illustrating examples are also
presented in this paper.
\end{quote}

\section{Introduction}

All considered groups are finite. Two subgroups $H$ and $K$ of a group $G$
are said to be permutable subgroups if $\left\langle H,K\right\rangle =HK$.
A subgroup $H$ of the group $G$ is permutable(quasinormal) in $G$ if $H$
permutes with every subgroup of $G$. If H permutes with all Sylow subgroups
of $G$ then $H$ is called $S$-permutable.

It should be remarked that permutable subgroups had been introduced in 1973
(see, \cite{ore}) by Ore who called them "\emph{quasinormal}"; and $S$%
-permutable subgroups were introduced in 1975 (see, \cite{Agrawal 1975} ) by
Agrawal who called them $S$-quasinormal ($\pi $-quasinormal).

A group $G$ is called a $PT$-group if, whenever $H$ is a permutable subgroup
of $G$ and $K$ is a permutable subgroup of $H$, then $K$ is a permutable
subgroup of $G$. Solvable\ $PT$-groups were first introduced and classified
by Zacher \cite{Zach}. Zacher's main result asserts that \textit{a group }$G$%
\textit{\ is a soluble }$PT$\textit{-group if and only if it has a normal
abelian Hall subgroup }$L$\textit{\ of odd order such that }$G/L$\textit{\
is a nilpotent modular group and elements of }$G$\textit{\ induce power
automorphisms in }$L$\textit{. }

Qusinormal subgroups have been studied by many authors, and generalized in
many different ways. Among the interesting generalizations is the concept of
$4$-quasinormality\ which was introduced by John Cossey and Stewart
Stonehewer in \cite{CosStew} as follows: A subgroup $H$ of the the group $G$
is called $4$\emph{-quasinormal} in $G$ \ and write ($H$ \emph{qn4} $G$) if $%
\left\langle H,K\right\rangle =HKHK$ for every \textit{cyclic} subgroup $K$
of $G$, and $H$ is \emph{strongly }$4$\emph{-quasinormal} in $G$ and we
write ($H$ \emph{sqn4} $G$) if $\left\langle H,K\right\rangle =HKHK$ for
\emph{every subgroup} $K$ of $G$. It is natural to thinck about groups in
which strong $4$-quasinormality is transitive relation. In this paper
introduce the class of $Sq4T$-groups as follows: a group $G$ is $Sq4T$-group
if strongly $4$-quasinormality is transitive in $G$. That is if $H$ \emph{%
sqn4} $K$ , and $K$ \emph{sqn4} $G$, then $H$ \emph{sqn4} $G$. Based on the
properties of $PT$-groups we descuss and develop some properties of $Sq4T$%
-groups.

\section{Preliminaries and Summary of the Results}

We start by introducing the concept $Sq4T$-groups and some related
consequences.

\begin{definition}
Let $H$ and $K$ be two subgroups of \ $G$. Then:

(1) Then $H$ \ is $4$-permutable with $K$ and we write ($H$ \emph{perm4} $K$%
) if $\left\langle H,K\right\rangle =HKHK$ ,

(2) $H$ is \emph{strongly }$4$\emph{-quasinormal} in $G$ and we write ($H$
\emph{sqn4} $G$) if $\left\langle H,K\right\rangle =HKHK$ for \emph{every
subgroup} $K$ of $G$.
\end{definition}

If $H$ is permutable subgroup in $G$ then $HK=KH$, for every $K$ subgroup of
$G$. This implies that $HKHK=H(KH)K=H(HK)K=HK=\left\langle H,K\right\rangle $%
. Therefor, every permutable subgroup is strongly\emph{\ }$4$-quasinormal.

\begin{remark}
\label{r1} A strongly\emph{\ }$4$-quasinormal subgroup need not be
permutable in $G$ as the following example shows.
\end{remark}

\begin{example}
\label{e1} Let $G=S_{3}$ be the symmetric group on three letters and let $%
H=\left\langle (1,2)\right\rangle $, $K=\left\langle (1,3)\right\rangle $.
Then $H$ does not permute with $K$ but $\left\langle H,K\right\rangle
=G=HKHK $. Similarly, $\left\langle H,\left\langle (2,3)\right\rangle
\right\rangle =S_{3}=H\left\langle (2,3)\right\rangle H\left\langle
(2,3)\right\rangle $. It is not difficult to check $4$-permutability of $H$
with the remaing subgroups of $S_{3}$. So, $H$ is strongly $4$-quasinormal
in $G$. In fact all subgroups of $S_{3}$ will be $4$-quasinormal in $S_{3}$.
\end{example}

\begin{definition}
A group $G$ is called $Sq4T$-group if strongly quasinormality is transitive
relation in $G$.
\end{definition}

\begin{remark}
\label{r2} Neither $PT$-group need to be $Sq4T$-group nor $Sq4T$-group need
be $PT$-group.
\end{remark}

The following is an example of $PT$-group that is not $Sq4T$-group.

\begin{example}
\label{e2} Let $G=D_{12}=\left\langle
r,s:r^{6}=s^{2}=1,r^{s}=r^{5}\right\rangle $ be the dihedral group of order $%
12$. If we let $L=\left\langle r^{2}\right\rangle $. Then \textit{\ }$%
L\backsimeq Z_{3}$\textit{\ }is Hall subgroup of $G$ of odd order such that $%
G/L\backsimeq Z_{2}\times Z_{2}$ is a nilpotent modular group and elements
of $G$ induce power automorphisms in $L$. So, $G$ is $PT$-group. To see that
$G$ is not $Sq4T$-group. consider the subgroup $H=\left\langle
r^{2},s\right\rangle $. To show that $H$ \emph{sqn4} $G$, we divide the
subgroups of $G$ into two types:

\textit{type 1:} Are all subgroups of $G$ that are not contained in $H$. If $%
T$ is subgroup of $G$ of type 1 then it is clear that $HTHT=G=\left\langle
H,T\right\rangle $.

\textit{type 2}: A subgroups $R$ of $G$ is of type 2 if $R$ is contained in $%
H$. For type 2 subgroups, note that $HRHR=H=\left\langle H,R\right\rangle $.
This shows that, $H$ \emph{sqn4} $G$. Now let $K=\left\langle s\right\rangle
$ be a subgroup of $G$. From Example \ref{e1} we know $H$ \emph{sqn4} $G$.
So, we have:

$K$ \emph{sqn4} $H$, and $H$ \emph{sqn4} $G$. But, $K$ \ is not strong $4$%
-quasinormal $\ $in $G$. To see that, $K$ is not strong $4$-quasinormal in $%
G $. Let $M=\left\langle sr\right\rangle \backsimeq Z_{2}$. Then, $KMKM$ \
has $8$ elements which could not be a subgroup of $G$. Therefore, $K$ is not
strong $4$-quasinormal in $G$. Hence, $G$ is not $Sq4T$-group\
\end{example}

\begin{example}
\label{Example 2} As an example of $Sq4T$-group that is not $PT$-group, we
may consider $G=D_{8}=\left\langle
r,s:r^{4}=s^{2}=1,r^{s}=r^{-1}\right\rangle $ the dihedral group of order $8$%
. It is clear that $\left\langle s\right\rangle $ is permutable in $%
\left\langle r^{2},s\right\rangle $, and $\left\langle r^{2},s\right\rangle $
is permutable in $G$. But $\left\langle s\right\rangle $ is not permutable
in $G$. Dividing the subgroups of D8 into two types (permutable and
non-permutable) and using an argument similar to that used in Example \ref%
{e1}. We get every subgroup of $G$ is strong $4$-quasinormal in $G$.
\end{example}

\begin{example}
\label{Example 3}A well known fact about permutable subgroup is Ore's
theorem (see \cite{Ore 1939})\ which asserts that \textit{in a finite group }%
$G$\textit{\ every permutable subgroup }$H$\textit{\ is subnormal in }$G$.
If we consider the alternating group $A_{5}$ the normal subgroups are $1$,
and $A_{5}$\ only. From Ore's theorem $1$ and $A_{5}$ are the only possible
permutable subgroups in $A_{5}$. Hence, $A_{5}$ is an example of $PT$-group
that is not solvable.
\end{example}

One of our results in this paper is given in the next theorem.

\begin{theorem}
\label{Theorem A} Every $Sq4T$-group is solvable.
\end{theorem}

However, in general, a solvable group need not be $Sq4T$-group. For an
example of solvable group that is not $Sq4T$-group we might consider the
group $D_{12}$ introduced in Example \ref{e1}. It also should be remarked
that a solvable $Sq4T$-group need not be supersolvable , as the following
example shows.

\begin{example}
\label{Example 4} Let $G=A_{4}$. We show that every subgroup of $G$ is
strong $4$-quasinormal in $G$. For that we divide the subgroups into three
sets as follows:

(1) $W_{1}=\{A_{4},P=\left\langle (1,2)(3,4),(1,3)(2,4)\right\rangle ,1\}$ -
the normal subgroups in $G$. These subgroups are all strong $4$-quasinormal
in $G$.

(2) $W_{2}=\{\left\langle (1,2)(3,4)\right\rangle ,\left\langle
(1,3)(2,4)\right\rangle ,\left\langle (1,4)(2,3)\right\rangle \}$, the
subgroups of order $2$. If $H_{1}$ and $H_{2}$ are any two distinct
subgroups of $W_{2}$. Then $H_{1}H_{2}=H_{2}H_{1}$, so $%
H_{1}H_{2}H_{1}H_{2}=H_{1}H_{2}=P=\left\langle H_{1},H_{2}\right\rangle $.
Note that $\left[ H_{1},H_{2}\right] =1$. So far, all subgroups in $W_{2}$
are $4$-permutable with subgroups of $W_{1}$ and $W_{2}$. To conclude their
strong $4$-quasinormality we have to show that, (in (3)), they are $4$%
-permutable with all subgroups of order $3$.

(3) $W_{3}=\{\left\langle (1,2,3)\right\rangle ,\left\langle
(1,2,4)\right\rangle ,\left\langle (1,3,4))\right\rangle ,\left\langle
(2,3,4))\right\rangle \}$, the subgroups of order $3$. We need to consider
two cases:

case 1. the subgroups of $W_{3}$ are $4$-permuting with each other. Let $%
K_{1}$ and $K_{2}$ be two distinct subgroups of $W_{3}$. A direct
computation shows that $K_{1}K_{2}K_{1}K_{2}=A_{4}=\left\langle
K_{1},K_{2}\right\rangle $. Moreover, $K_{1}K_{2}\neq K_{2}K_{1}$ implies $1<%
\left[ K_{1},K_{2}\right] \leq \left[ A_{4},A_{4}\right] =P$. That is $\left[
K_{1},K_{2}\right] =P$.

case 2. If $H\in W_{2}$ and $K\in W_{3}$, then we have $HKHK=A_{4}=\left%
\langle H,K\right\rangle $; hence every subgroup of $G$ is strongly $4$%
-quasinormal in $G$. That implies $G$ is $Sq4T$-group which is not
supersolvable.
\end{example}

\section{The Proofs}

In this section we prove Theorem \ref{Theorem A}. For that, we develop some
properties of strong $4$-quasinormal subgroups which in their own are
interesting facts.

We start with the following well-known properties about commutator subgroups.

\begin{lemma}
\label{Lemma 1} Let $H$ and $K$ be subgroups of $G$, Then:

(1) $\left[ H,K\right] \trianglelefteq \left\langle H,K\right\rangle $,

(2) $\left[ H,K\right] =\left[ K,H\right] $,

(3) $H$ normalizes $K$ if and only if $\left[ H,K\right] \leq K$.
\end{lemma}

\begin{lemma}
\label{Lemma 2} Let $H\leq K$ be subgroups of $G$. If $H$ \ $sq4$ \ $G$,
then $H$ \ $sq4$ $K$
\end{lemma}

\begin{proof}
This follows from the definition of strong $4$-quasinormality.
\end{proof}

\begin{lemma}
\label{Lemma 3} Let $N\leq H$ \ be two subgroups of $G$ such that $%
N\trianglelefteq G$. If $H$ $sq4$ $G$. Then $H/N$ $sq4$ $G/N$.
\end{lemma}

\begin{proof}
Assume that $H$ $sq4$ $G$, and let $K/N$ be any subgroup in $G/N$. Then $%
(H/N)(K/N)(H/N)(K/N)=(HKHK)/N=\left\langle H,K\right\rangle /N=\left\langle
H/N,K/N\right\rangle $. Hence, $H/N$ $sq4$ $G/N$.
\end{proof}

\begin{proposition}
\label{Propos 1} If $G$ is $Sq4T$-group and $H$ is subgroup of $G$. Then $H$
is $Sq4T$-group.
\end{proposition}

\begin{proof}
(Of Proposition \ref{Propos 1}.) Assume that the proposition is false and
let $G$ \ be a group of minimal order for which the proposition does not
hold. Let $H$ \ be a proper subgroup of $G$ that is not $Sq4T$-group. For $H$
we consider two cases:

\textit{Case 1}. $H$\emph{\ is }$sq4$\emph{\ in }$G$. If $H_{2}$ and $H_{1}$
are subgroup of $H$ such that: $H_{2}$ $sq4$ $H_{1}$, $H_{1}$ $sq4$ $H$. We
need to have $H_{2}$ $sq4$ $H$. But, $G$ is $Sq4T$-group, so $H_{1}$ $sq4$ $H
$ $sq4$ $G$ implies $H_{1}$ $sq4$ $G$. Now, $H_{1}$ $sq4$ $H_{2}$ $sq4$ $G$
implies $H_{1}$ $sq4$ $G$. From Lemma \ref{Lemma 2}. We have $H_{1}$ $sq4$ $H
$. Therefore, $H$ is $Sq4T$-group. Contradiction.

\textit{Case 2}. $H$\emph{\ is not }$sq4$\emph{\ in }$G$, let $H$ be a
subgroup of smallest order with this property. Let $M$, and $N$ be two
subgroups of $H$ such that $M$ $sq4$ $N$ $sq4$ $H$, and suppose $M$ is not $%
sq4$ $H$. Then $H$ \ has subgroup $L$ such that $LMLM\neq \left\langle
L,M\right\rangle $. Now consider $\left\langle L,N\right\rangle =LNLN$.

If \ $\left\langle L,N\right\rangle =H$, then $\left[ L,N\right]
\trianglelefteq H$. Since $N$ $sq4$ $H$, then $\left[ L,M\right]
\trianglelefteq H$. Therefore, $M$ is $sq4$ $H$, and $H$ is $Sq4T$-group
Contradiction. Therefore, $\left\langle L,N\right\rangle \neq H$. From Lemma %
\ref{Lemma 2}. We have $M$ $sq4$ $\left\langle L,N\right\rangle $ and $L\leq
\left\langle L,N\right\rangle $. It is clear that $\left\langle
L,N\right\rangle $ is not $Sq4T$-group and $\left\vert \left\langle
L,N\right\rangle \right\vert $ is less that $\left\vert H\right\vert $.
Contradiction and the proposition is proved.
\end{proof}

\begin{proof}
(Of Theorem \ref{Theorem A}) By Contradiction. Let $G$ be a $Sq4T$- group of
smallest order that is not solvable. From Proposition \ref{Propos 1}. we
have every subgroup of $G$ is $Sq4T$-group. By induction all proper
subgroups of $G$ are solvable. So, $G$ is minimal non-solvable group. For $G$
we consider two cases:

\textit{Case 1}. If $G$ is not simple. Let $M$ be a maximal normal subgroup
of $G$. From Proposition \ref{Propos 1}. $M$ \ is $Sq4T$-group. By induction
$M$ must be solvable. From Lemma \ref{Lemma 3} $G/M$ is $Sq4T$-group. By
induction, $G/M$ is solvable. Hence, $G$ is solvable. Contradiction.

\textit{Case 2}. If $G$ is simple group. Let $M_{1}$, and $M_{2}$ be proper
subgroups of greatest order such that $M_{i}$ $sq4$ $\ G$ \ ($1\leq i\leq 2$%
). Then it is clear that $\left[ M_{1},M_{2}\right] =G$. Now the simplicity
of $G$ prevent $M_{i}$ form having any proper non trivial $sq4$ subgroups.
So, $G$ \ can not be $Sq4T$-group. Contradiction. Therefore, $G$ must be
solvable.
\end{proof}

\end{document}